\documentclass[11pt, a4paper, twoside]{amsart}
\usepackage{mathrsfs, tabularx}

\usepackage {amsthm, epsfig, amsmath, amssymb, amscd, float, latexsym, times, epic, eepic}

\theoremstyle{plain}
\newtheorem{theorem}{Theorem}[section]
\newtheorem{proposition}[theorem]{Proposition}

\theoremstyle{definition}
\newtheorem{definition}[theorem]{Definition}

\newcommand {\Set}[1] {\mathbb{#1}}
\newcommand{\setR}[0]{\Set{R}}

\newcommand{\slaz}[0]{{\setminus\{0\}}}

\newcommand{\pd}[2]{\frac{\partial #1}{\partial #2}}

\newcommand{\vfield}[1]{{\mathfrak X}( #1)}

\title{Electromagnetic fields from contact forms}
\author[Dahl]{Matias F. Dahl}
\address{
Matias F. Dahl, Institute of Mathematics, P.O.Box
1100, 02015 Helsinki University of Technology, Finland }
\urladdr{http://www.math.tkk.fi/\textasciitilde{}fdahl/}

 %EM-theory general
\date{\today}

\begin{document}
\begin{abstract}
  In this short note we prove that every contact form on a
  $3$-manifold $M$ induces a solution to Maxwell's equations on $M$.
\end{abstract}

\maketitle

\section{Introduction}
The main result of this note is the following theorem. It shows that
every contact form on a $3$-manifold induces a time dependent solution
to the source-less Maxwell's equations on the same manifold.

\begin{theorem} 
\label{mainThm}
Suppose $\alpha\in \Omega^1(M)$ is a contact form on a $3$-manifold
$M$ and $\omega\in \setR\slaz$.  Then there exists a $1$-form
$\beta\in \Omega^1(M)$ such that
\begin{eqnarray*}
  E(x,t) &=& \operatorname{Re} \{ \alpha e^{i \omega t}\},   \\
  H(x,t) &=& \operatorname{Im} \{ \beta e^{i \omega t}\},   \quad\quad (x,t)\in M\times \setR 
\end{eqnarray*}
is a solution to the source-less Maxwell's equations in an
electromagnetic media determined by $\alpha$.
\end{theorem}

The converse question was studied in \cite{Dahl2004}: If we start with
an electromagnetic field, can we extract contact structures from it?
This question seem to be much more difficult.
Particular examples of such contact structures can be found in 
\cite{Dahl2004}. As an example, the overtwisted contact structures on 
$\setR^3$ are induced by plane-wave solutions to Maxwell's equations. 

%(For precise notation and terminology, see Section \ref{mainSec}.)

The proof of Theorem \ref{mainThm} is a relatively direct consequence
of an observation of Chern and Hamilton in \cite{ChernHamilton}.
Namely: every contact form on a $3$-manifold has an adapted
Riemann metric (Proposition \ref{ChHa} below).
To prove Theorem \ref{mainThm} we use Proposition \ref{ChHa} to obtain
a Riemann metric $g$ adapted to the contact form $\alpha$. The
advantage of this metric is that it maps the contact $1$-form
$\alpha\in \Omega^1(M)$ into a Beltrami vector field $\alpha^\sharp\in
\vfield{M}$. We can then define electromagnetic media on $M$ by
setting $g_\varepsilon = g_\mu=g$, and after this the proof is
essentially a direct calculation. This approach was also used in
\cite{EtGh2000} by Etnyre and Ghrist to study connections between
contact geometry and hydrodynamics.
Let us point out that the relation between Beltrami vector fields and
electromagnetics is well known.  See for example \cite{Dahl2004,
  Lak, LiSiTrVi:1994}. For hydrodynamics, see \cite{EtGh2000}.

\section{Terminology and proof of Theorem \ref{mainThm}}
\label{mainSec}
We assume that $M$ is a smooth $3$-manifold. That is, $M$ is
a Hausdorff, second countable, topological space that is locally
homeomorphic to $\setR^3$ with smooth transition maps. All objects are
smooth where defined; $k$-forms are denoted by
$\Omega^k(M)$, vector fields are denoted by $\vfield{M}$, and
functions are denoted by $C^\infty(M)$.  By $\Omega^k(M)\times \setR$ we
denote the set of $k$-forms that depend on a parameter $t\in \setR$.

\subsection{Maxwell's equations} We will use differential forms to
write Maxwell's equations on manifold $M$. See \cite{BH1996,
Bossavit:2001}.  For field quantities $E,H\in \Omega^1 (M)\times
\setR$ and $D,B\in \Omega^2 (M)\times \setR$ the \emph{sourceless Maxwell equations} read
\begin{eqnarray}
\label{max1}
  dE &=& - \pd{B}{t}, \\  
\label{max2}
  dH &=& \pd{D}{t}, \\
\label{max3}
  dD &=& 0, \\
\label{max4}
  dB &=& 0, 
\end{eqnarray}
and the \emph{constitutive equations} read
\begin{eqnarray}
\label{con1}
  D  &=&\ast_\varepsilon E, \\
\label{con2}
  B &=& \ast_\mu H, 
\end{eqnarray}
where $\ast_\varepsilon$ and $\ast_\mu$ are Hodge star operators
corresponding to two Riemann metrics $g_\varepsilon$ and $g_\mu$,
respectively. On a $3$-manifold the \emph{Hodge star operator} $\ast$
is the map
$\ast\colon \Omega^p(M)\to \Omega^{3-p}(M)$ ($p=0,\ldots, 3$)
that acts on basis elements of $\Omega^p(M)$ as
\begin{eqnarray*}
\label{hodgedef}
\ast(dx^{i_1} \wedge \cdots \wedge dx^{i_p})\!\!\!\! &=& \!\!\!\!\frac{\sqrt{|g|}}{(3-p)!} g^{i_1 l_1}\cdots  g^{i_p l_p} \varepsilon_{l_1 \cdots l_p\, l_{p+1} \cdots l_n} dx^{l_{p+1}}\wedge \cdots \wedge  dx^{l_{n}}.
\end{eqnarray*}
Here $g=g_{ij}dx^i\otimes dx^j$, $|g|=\det g_{ij}$,
$g^{ij}$ represent the $ij$th entry of $(g_{ij})^{-1}$, and
$\varepsilon_{l_1\cdots l_n}$ is the \emph{Levi-Civita permutation
symbol}. On a $3$-manifold we always have $\ast^2 = \operatorname{Id}$. 

For a Riemann metric $g$, let $\sharp$ and $\flat$ be the musical
isomorphisms $\sharp\colon T^\ast M\to TM$ and $\flat \colon TM\to
T^\ast M$.

\subsection{Contact geometry and Beltrami fields} 
%In general, contact structures can exist on odd dimensional manifolds
%of dimension $\ge 3$. However, let us here restrict our attention only
%to contact structures on $3$-manifolds that are globally represented 
%by a $1$-form. %
A \emph{contact form} on a $3$-manifold is a $1$-form $\alpha\in
\Omega^1(M)$ such that $\alpha \wedge d\alpha$ is never zero
\cite{Geiges2008}. By Frobenius theorem \cite{Boothby}, a form
$\alpha\in \Omega^1(M)$ is a contact form if and only if the plane
field
$$
  \operatorname{ker \alpha} = \{ v\in TM : \alpha(v)=0\}
$$ 
is nowhere integrable. For a contact form $\alpha$, the pair
$(M,\operatorname{ker}\alpha)$ is called a \emph{contact
  structure}. By definition, every contact form induces an orientation
on $M$. 

\begin{definition}[Adapted Riemann metric]
\label{adapt}
A contact form $\alpha\in\Omega^1(M)$ and a Riemann metric $g$ are 
\emph{adapted} if
\begin{eqnarray}
\label{adaptMetric}
  d\alpha = 2\ast \alpha, \quad g(\alpha^\sharp,\alpha^\sharp)=1,
\end{eqnarray}
where $\ast$ is the Hodge star operator induced by $g$.
\end{definition}

Proposition \ref{ChHa} is due to Chern and Hamilton
\cite{ChernHamilton} (who also studied conditions on curvature for
$g$). 
Direct proofs can be found in \cite{Kom2006, EtGh2000}.
%For contact structures on manifolds of dimension $\ge 5$, one can prove
%that the metric can not be flat \cite[p. 115]{Blair1976}. 

\begin{proposition}[Chern, Hamilton -- 1984] 
\label{ChHa}
Every contact form on a $3$-manifold has an
(non-unique) adapted Riemann metric.
\end{proposition}

To understand the relevance of adapted Riemann metrics, let us define
the curl of a vector field $X\in \vfield{M}$ as the unique vector field
$\nabla\times X\in \vfield{M}$ determined by
$$
  (\nabla\times X)^\flat = \ast d( X^\flat).
$$
By setting $X=\alpha^\sharp$, conditions \eqref{adaptMetric} read
$$
  \nabla\times X = 2 X, \quad g(X,X)=1. 
$$
That is, an adapted metric turns the contact form into a non-vanishing
Beltrami vector field.  A \emph{Beltrami vector field} is a vector
field $F\in \vfield{M}$ such that $\nabla \times F = f F$ for some
function $f\in C^\infty(M)$ \cite{EtGh2000}.
In this note we only work with forms; we study contact forms and
electromagnetic fields, and both of these are most naturally
represented using forms, and not vector fields. It is therefore
motivated to work with Beltrami $1$-forms instead of Beltrami vector
fields. This motivates the next definition \cite{Dahl2004}.

\begin{definition}[Beltrami $1$-form]
A \emph{Beltrami $1$-form} for a Riemann metric with Hodge star operator
$\ast$ is a $1$-form $\alpha \in \Omega^1(M)$ such that
\begin{eqnarray}
\label{BeltramiEq}
  d\alpha = f \ast \alpha
\end{eqnarray}
for some $f\in C^\infty(M)$. Moreover, $\alpha$ is a \emph{rotational 
Beltrami $1$-form} if $f$ is nowhere zero.
\end{definition}

The next proposition will play a key role in the proof of Theorem
\ref{mainThm}. This proposition shows that non-vanishing rotational
Beltrami $1$-forms and contact $1$-forms are essentially in one-to-one
correspondence. 
Proposition \ref{belCoEq} is essentially a restatement (using $1$-forms)
of the result used by Etnyre and Ghrist in \cite{EtGh2000} to study
connections between contact geometry and hydrodynamics.

\begin{proposition}[Etnyre, Ghrist -- 2000]
\label{belCoEq}
Let $\alpha \in \Omega^1(M)$.
\begin{enumerate}
\item If $\alpha$ is a rotational Beltrami $1$-form that in nowhere
  zero, then $\alpha$ is a contact form.
\item If $\alpha$ is a contact form, and $f$ is a strictly
  positive function $f\in C^\infty(M)$, then there exists a
  (non-unique) Riemann metric on $M$ such that equation
  \eqref{BeltramiEq} holds.  In this case, $\alpha$ is a rotational
  Beltrami $1$-form.
\end{enumerate}
\end{proposition}

\begin{proof}
For \emph{(i)} we have
$
  \alpha \wedge d\alpha = f \alpha \wedge \ast \alpha. 
$ 
By Proposition 6.2.12 in \cite{MTAA}, 
\begin{eqnarray*}
  \alpha\wedge \ast \alpha &=& g(\alpha^\sharp, \alpha^\sharp) dV,
\end{eqnarray*} 
where $dV$ is the Riemann volume form. Hence $\alpha \wedge d\alpha$
vanishes only if $\alpha$ or $f$ vanishes.  For \emph{(ii)}, let us
first note that if $g$ and $\widetilde g$ are metrics such that
$\widetilde g = \mu\, g$ for some strictly positive function $\mu\in
C^\infty(M)$, then corresponding Hodge star operators $\ast,
\widetilde \ast \colon \Omega^1(M)\to \Omega^2(M)$ satisfy $\widetilde
\ast = (\mu)^{1/2} \,\ast$.
For the proof, let $g$ be a Riemann metric such that
\eqref{adaptMetric} holds. A suitable Riemann metric is then
$\widetilde g = 4/f^2 g$. In fact, for induced Hodge operators $\ast$
and $\widetilde \ast$, we obtain
$$
  f\, \widetilde \ast \alpha =  2 \ast \alpha = d\alpha
$$
and equation \eqref{BeltramiEq} follows.  
\end{proof}

\begin{proof}[Proof of Theorem \ref{mainThm}]
By Proposition \ref{belCoEq} \emph{(ii)}, there exists a Riemann
metric $g$ adapted to $\alpha$ such that
\begin{eqnarray}
\label{eqa}
  d\alpha &=& \vert \omega \vert \ast \alpha
\end{eqnarray}
Suitable electromagnetic media is given by $g_\varepsilon =
g_\mu=g$. Let $\beta\in \Omega^1(M)$ be the unique $1$-form determined
by
\begin{eqnarray}
\label{dap}
% d\alpha &=& -  \omega \ast_\mu \beta.
 d\alpha &=& -  \omega \ast \beta.
\end{eqnarray}
Equation \eqref{max1} follows.  Equation \eqref{eqa} implies that
\begin{eqnarray}
  \ast d  \ast d \alpha &=& \omega^2 \alpha.
\end{eqnarray}
Hence
\begin{eqnarray}
\label{dbp}
  d\beta &=&  - \omega \ast \alpha,
\end{eqnarray}
and equation \eqref{max2} follows.
Equation \eqref{dbp} implies that $d\ast \alpha=0$ and equation
\eqref{max3} follows.  Similarly, equation \eqref{max4} follows by
equation \eqref{dap}.
\end{proof}

\subsection*{Acknowledgements}
The author gratefully appreciates the finanical support provided by
the Academy of Finland Centre of Excellence programme 213476, the
Institute of Mathematics at the Helsinki University of Technology, and
Tekes project MASIT03 --- Inverse Problems and Reliability of Models.

%\bibliographystyle{amsalpha}
%\bibliography{dis}
\providecommand{\bysame}{\leavevmode\hbox to3em{\hrulefill}\thinspace}
\providecommand{\MR}{\relax\ifhmode\unskip\space\fi MR }
\providecommand{\MRhref}[2]{%
  \href{http://www.ams.org/mathscinet-getitem?mr=#1}{#2}
}
\providecommand{\href}[2]{#2}

\end{document}